\documentclass[11pt,a4paper,leqno]{amsart}
\usepackage{graphics}
\usepackage{epsfig}
\usepackage{yfonts}
\usepackage{graphicx}
\usepackage[matrix,arrow,tips,curve,ps]{xy}
\usepackage{enumerate, amssymb}
\usepackage{hyperref}
\usepackage{comment, paralist}
\usepackage{xcolor}

\usepackage{lmodern}
\usepackage[T1]{fontenc}
\usepackage[utf8]{inputenc}

\def\bdi{\begin{diagram}}
\def\edi{\end{diagram}}

\newtheorem{thm}{Theorem}
\newtheorem{cor}[thm]{Corollary}
\newtheorem{lem}[thm]{Lemma}
\newtheorem{prop}[thm]{Proposition}

\theoremstyle{definition}
\newtheorem{defi}[thm]{Definition}
\newtheorem{defis}[thm]{Definitions}
\newtheorem{conj}[thm]{Conjecture}

\newtheorem{conv}[thm]{Convention}
\newtheorem{nota}[thm]{Notation}
\newtheorem{rem}[thm]{Remark}
\newtheorem{rems}[thm]{Remarks}
\newtheorem{exa}[thm]{Example}
\newtheorem{exas}[thm]{Examples}
\newtheorem{claim}[thm]{Claim}
\newcommand{\rien}[1]{}

\renewcommand{\P}{\mathbf{P}}
\newcommand{\C}{\ensuremath{\mathbf{C}}}
\newcommand{\D}{\ensuremath{\mathbf{D}}}

\newcommand{\Z}{\ensuremath{\mathbf{Z}}}

\newcommand{\cI}{{\ensuremath{\mathcal{I}}}}

\def\deg{\mathop{\rm deg}}

\def\Pic{\mathop{\rm Pic}}

\renewcommand{\epsilon}{\varepsilon}
\renewcommand{\phi}{\varphi}

\newcommand{\bnum}{\begin{enumerate}}
\newcommand{\enum}{\end{enumerate}}

\renewcommand{\leq}{\leqslant}
\renewcommand{\geq}{\geqslant}

\newcommand{\brem}{\begin{rem}}
\newcommand{\brems}{\begin{rems}}
\newcommand{\erem}{\end{rem}}
\newcommand{\erems}{\end{rems}}
\newcommand{\bexa}{\begin{exa}}
\newcommand{\bexas}{\begin{exas}}
\newcommand{\eexa}{\end{exa}}
\newcommand{\eexas}{\end{exas}}
\newcommand{\bdefi}{\begin{defi}}
\newcommand{\edefi}{\end{defi}}
\newcommand{\bdefis}{\begin{defis}}
\newcommand{\edefis}{\end{defis}}
\newcommand{\bcor}{\begin{cor}}
\newcommand{\ecor}{\end{cor}}
\newcommand{\blem}{\begin{lem}}
\newcommand{\elem}{\end{lem}}
\newcommand{\bconv}{\begin{conv}}
\newcommand{\econv}{\end{conv}}
\newcommand{\bconj}{\begin{conj}}
\newcommand{\econj}{\end{conj}}
\newcommand{\bprop}{\begin{prop}}
\newcommand{\eprop}{\end{prop}}
\newcommand{\bthm}{\begin{thm}}
\newcommand{\ethm}{\end{thm}}
\newcommand{\bnota}{\begin{nota}}
\newcommand{\enota}{\end{nota}}
\newcommand{\bsit}{\begin{sit}}
\newcommand{\esit}{\end{sit}}
\newcommand{\be}{\begin{eqnarray}}
\newcommand{\ee}{\end{eqnarray}}
\newcommand{\bproof}{\begin{proof}}
\newcommand{\eproof}{\end{proof}}
\def\ba{\begin{array}}
\def\ea{\end{array}}

\newcommand{\K}{\mathcal{K}}
\newcommand{\restr}[2]{\left. #1 \right| _{#2}}

\renewcommand{\O}{\mathcal{O}}
\renewcommand{\H}{\mathrm{H}}

\newcommand{\sev}[2]{V^{#1,#2}} 
\newcommand{\csev}[2]{\smash{\overline{V}}^{#1,#2}} 
\newcommand{\gsev}[2]{V^{#1}_{#2}} 
\newcommand{\stsev}[2]{V^{\smash{#1,#2}}_{\mathrm{st}}}
\newcommand{\cstsev}[2]{\smash{\overline V}^{\smash{#1,#2}}_{\mathrm{st}}}

\title[]{On the irreducibility of Severi varieties on $K3$ surfaces}

\author{C.~Ciliberto}
\address{Dipartimento di Matematica, Universit\`a degli
Studi di Roma ``Tor Vergata'', Via della Ricerca Scientifica,
00133 Roma, Italy} \email{cilibert@mat.uniroma2.it}

\author{Th.~Dedieu}
\address{Institut de Mathématiques de Toulouse~; UMR5219.
Université de Toulouse~; CNRS.
UPS IMT, F-31062 Toulouse Cedex 9, France.} 
\email{thomas.dedieu@math.univ-toulouse.fr}

\makeatletter
\def\@acks{
The authors were members of project FOSICAV, which has received
funding from the European Union's Horizon 2020 research and innovation
programme under the Marie Sk{\l}odowska-Curie grant agreement
No~652782.
The first author has been supported by the Italian MIUR Project PRIN
2010--2011 ``Geometria delle varietà algebriche'' and by GNSAGA of
INdAM;
he also acknowledges the MIUR Excellence Department Project awarded to
the Department of Mathematics, University of Rome Tor Vergata, CUP
E83C18000100006.
}
\makeatother

\date{}
\begin{document}

\setdefaultenum{(i)}{}{}{}

\begin{abstract} 
Let $(S,L)$ be a polarized $K3$ surface of genus $p \geq 11$ such
that $\Pic(S)=\Z[L]$, and $\delta$ a non-negative integer.
We prove that if $p\geqslant 4\delta-3$, then the Severi
variety of $\delta$-nodal curves in $|L|$ is
irreducible.
\end{abstract}

\maketitle

\vfuzz=2pt
\thanks{}

\section{Introduction}

Given a polarized surface $(S,L)$ and an integer $\delta \geq 0$,
the \emph{Severi variety} $\sev L \delta$ is the parameter
space for irreducible, $\delta$-nodal curves in the linear system
$|L|$ (see \S~\ref{s:severi}).
This text is dedicated to the proof of the following result:
\begin{thm}
\label{t:main}
Let $(S,L)$ be a primitively polarized $K3$ surface of genus $p \geq
11$ such that $\Pic(S)=\Z [L]$,
and $\delta$ a non-negative integer such that $4\delta-3 \leq p$.
The Severi variety $\sev L \delta$ is irreducible.
\end{thm}

It had already been proven by Keilen \cite{keilen} that in the
situation of Theorem~\ref{t:main}, for all integer $k\geq 1$ the 
Severi variety $\sev {kL} \delta$ is irreducible if
\begin{equation*}
\delta <
\frac {6(2p-2)+8}
{\bigl(11(2p-2)+12\bigr)^2}
\cdot k^2 \cdot (2p-2)^2
\qquad \left(
\sim_{p \to \infty} \frac {12}{121}\cdot k^2\cdot p
\right),
\end{equation*}
and later by Kemeny \cite{kemeny} that the same holds if
$\delta \leq \frac 1 6 \bigl(2+k(p-1)\bigr)$.
Our result is valid only in the case $k=1$, i.e., for curves in the
\emph{primitive class}, but in this case our condition is better.
In a slightly different direction, we have proven some time ago in
\cite{CD} that 
the \emph{universal families} of the $\sev L \delta$'s are irreducible
for all $\delta$ ($\delta = p$ included) if $3 \leq p \leq 11$ and
$p\neq 10$.

Kemeny's result is based on the observation that for any smooth
polarized surface $(S,L)$, the Severi variety $\sev L \delta$ is
somehow trivially irreducible if $L$ is $(3\delta-1)$-very ample: 
Indeed, in this case the curves in $|L|$ with nodes at
$p_1,\ldots,p_\delta$ form a dense subset of a projective space of
constant dimension
for \emph{any} set of pairwise distinct points
$p_1,\ldots,p_\delta$.
Kemeny then applies a numerical criterion for $n$-very ampleness
on $K3$ surfaces due to Knutsen \cite{knutsen}.

The central idea of the present article is close in spirit to Kemeny's
observation, to the effect that provided $\dim |L| \geq 3\delta$, 
the curves in $|L|$ with nodes at
$p_1,\ldots,p_\delta$ should form in nice circumstances a dense subset
of a projective space of constant dimension
 for a \emph{general} choice of $\delta$ pairwise disjoint points.
It is indeed so for curves in the primitive class of
a $K3$ surface, thanks to a result of Chiantini and the first-named
author, see Proposition~\ref{prop:dk3}.
One thus gets a distinguished irreducible component of the Severi
variety $\sev L \delta$ which we call its \emph{standard component}. 
For any other irreducible component $V$, the nodes of the members of
$V$ sweep out a locus of positive codimension $h_V$ in the Hilbert
scheme $S^{[\delta]}$, see Section~\ref{S:standard};
we call $h_V$ the excess of $V$.

Our applications then rely on the observation
that, in the $K3$ situation of Theorem~\ref{t:main}, for all $C \in V$ the preimage of the nodes defines a linear
series of type $g^h_{2\delta}$ on the normalisation of $C$
(see Lemma~\ref{l:dim-deltatilde}),
together with some recent results in \cite{CK} and \cite{KLM} 
(Theorems~\ref{t:CK} and \ref{t:KLM} respectively) 
which give some control on the families
of linear series that may exist on the normalisations of primitive
curves on $K3$ surfaces.
The latter results hold only for curves in the primitive class, and
this is the main obstruction to carry out our approach in the
non-primitive situation.

One may for instance give a two-lines proof of irreducibility in the
range $p \geq 5\delta-3$, as follows.
Assume by contradiction that there is a non-standard irreducible
component $V$ of the Severi variety $\sev L \delta$. Then for all $C
\in V$ the normalisation of $C$ has a $g^1_{2\delta}$. 
By \cite{KLM} this implies
$\dim(V) = p-\delta \leq 4\delta-2$, 
which is impossible in the range under consideration.

We obtain the better bound in Theorem~\ref{t:main} by proving the
estimate $h_V >2$ for all non-standard components of $\sev L \delta$. 
This is done in Section~\ref{S:estim-h} by a careful study of the
singularities of curves in the intersection of the standard component
with a hypothetical non-standard component, which we are again able to
control thanks to Brill--Noether theoretic results for singular curves
on $K3$ surfaces.

\medskip
This work originates from the Oberwolfach Mini-Workshop: Singular
Curves on $K3$ Surfaces and Hyperkähler Manifolds. We thank all the
participants for the friendly atmosphere and stimulating discussions.

\section{Preliminaries}

\subsection {Severi varieties}
\label{s:severi}
We work over $\C$ throughout the text.
We denote by $\mathcal K_p$ the irreducible, 19-dimensional stack of
primitively polarized $K3$ surfaces $(S,L)$ of genus $p\geqslant 2$,
i.e., $S$ is a compact, complex surface with $h^ 1(S,\mathcal O_S)=0$
and $\omega_S\cong \mathcal O_S$, and $L$ a big and nef, primitive
line bundle on $S$ with $L^ 2=2p-2$, hence $\dim(|L|)=p$. The
\emph{arithmetic genus} of the curves $C\in |L|$ is $p_a(C)=p$.

In this paper we will often assume that $\Pic (S)=\Z [L]$, which is
the case if $(S,L)\in \mathcal K_p$ is very general, so that $L$ is
globally generated and ample, and very ample if $p\geqslant 3$.

For any non-negative integer $g\leqslant p$, we consider the locally
closed subset $\gsev L g$ of $|L|$ consisting of curves $C\in |L|$ of
\emph{geometric genus} $p_g(C)=g$, i.e., curves $C$ whose
normalization has genus $g$ (see \cite [\S~1.2] {DeSe}). We will set
$\delta=p-g$, which is usually called the \emph{$\delta$-invariant} of
the curve.

\begin{prop}[see {\cite[Proposition~4.5]{DeSe}}]
\label{prop:dim}
Every irreducible component of  $\gsev L g$ has dimension $g$.
\end{prop}

For every non-negative integer $\delta\leqslant p$, we will denote by
$\sev L \delta$ the \emph{Severi variety}, i.e., the locally closed
subset of $|L|$ consisting of curves with $\delta$ nodes and no other
singularities, whose geometric genus is $g=p-\delta$. The following is
classical:

\begin {prop} [see {\cite[\S 3--4]{DeSe}}]
\label {prop:sevv}
The Severi variety $\sev L \delta$, if not empty, is smooth and pure of
dimension $g$. More precisely, if $C\in \sev L \delta$, and $\Delta$ is
the set of nodes of $C$, then the projective tangent space to
$\sev L \delta$ at $C$ in $|L|$ is the $g$-dimensional linear system
$|L(-\Delta)|:=\P (\H^ 0(S, L \otimes \mathcal I_{\Delta,S}))$ of
curves in $|L|$ containing $\Delta$.
\end{prop}

\noindent
It is indeed true that the Severi varieties of a general primitively
polarized $K3$ surface are non-empty.

\begin{prop}[see \cite {XCh99}] 
If $(S,L)\in \mathcal K_p$ is general,
then $\sev L \delta$ is not empty for every non-negative integer
$\delta\leqslant p$.
\end{prop}

By Propositions \ref {prop:dim} and \ref {prop:sevv}, each irreducible
component of $\sev L \delta$ is dense in a component of $V^ L_{g}$. 
Xi Chen \cite{XCh-pre16} has shown that moreover if $g>0$, then 
$\sev L \delta$ is dense in $\gsev L g$ for general $(S,L) \in \K_p$.
We shall need the following weaker result, in which however the
generality assumption is explicit.\footnote
{Actually, the assumption in \cite[Proposition~4.8] {DeSe} is that
$(S,L)$ be very general; it is straightforward to check that the
condition $\Pic(S)=L$ is indeed sufficient for the proof in 
\cite{DeSe}.}

\begin{prop}[{\cite[Proposition~4.8] {DeSe}}] 
\label{prop:sev}
Let $(S,L)\in\mathcal K_p$ be such that $\Pic (S)=\Z [L]$. If
$2\delta <p$, then $\sev L \delta$ is dense in $\gsev L g$.
\end{prop}

\subsection{Local structure of Severi varieties}
The following is a restatement of the well-known fact that the nodes
of a nodal curve on a $K3$ surface may be smoothed independently.
It is a consequence of Proposition~\ref{prop:sevv}. 

\begin{prop}
\label{prop:ext}
Let $(S,L)\in \mathcal K_p$, 
$\delta < \epsilon$ be two non-negative integers, 
and $V$ be an irreducible component of
$\sev L \epsilon$. Consider a curve $C\in V$, and let
$\{p_1,\ldots,p_\epsilon\}$ be the set of its nodes.
Then:\\
\begin{inparaenum} 
\item [(i)] the Zariski closure $\csev L \delta$ of $\sev L
  \delta$ contains $V$;\\ 
\item [(ii)] locally around $C$, $\csev L \delta$ consists of
$\binom \epsilon \delta$ analytic sheets $\mathcal V_\mathfrak d$, which
are in $1:1$ correspondence with the subsets 
$\mathfrak d \subset \{p_1,\ldots,p_\epsilon\}$ of order $\delta$,
and such that when the general point $C'$ of $\mathcal
V_\mathfrak d$ specializes at $C$, the set of $\delta$ nodes of
$C'$ specializes at $\mathfrak d$;\\
\item [(iii)] for each such $\mathfrak d$, the sheet $\mathcal
V_\mathfrak d$ is smooth at $C$ of dimension $p-\delta$, relatively
transverse to all other similar sheets.\footnote
{in the sense that for all $\mathfrak d'$ of cardinality $\delta$,
the sheets $\mathcal V_{\mathfrak d}$ and $\mathcal V_{\mathfrak d'}$
intersect exactly along the local sheet 
$V_{\mathfrak d \cup \mathfrak d'}$ of 
$\csev L {|{\mathfrak d \cup \mathfrak d'}|}$ at $C$, 
and their respective tangent spaces at $C$ intersect exactly along the
tangent space of
$V_{\mathfrak d \cup \mathfrak d'}$ at $C$.}
\end{inparaenum} 
\end{prop}

\noindent
As an immediate consequence, we have:

\begin{cor}\label{cor:incl}  
Let $(S,L)\in \mathcal K_p$ and let $V$, $V'$ be irreducible
components of $\sev L \delta$ and $\sev L {\delta'}$, with
$\delta\leqslant \delta'$. If $V'$ intersects the Zariski closure
$\overline V$ of $V$, then $V'\subset \overline V$.
\end{cor}

\subsection{Brill--Noether theory of curves on K3 surfaces} 
We will use the following results.

\begin{thm}[{\cite[Theorem~5.3 and Remark~5.6]{KLM}}]
\label{t:KLM}
Let $(S,L)$ be such that $\Pic(S)=\Z[L]$, and
$V \subset \gsev L g$ a non-empty reduced scheme.
Let $k$ be a positive integer.
Assume that for all $C \in V$, there exists a $g^1_k$ on the
normalisation $\tilde C$ of $C$.
Then one has
\begin{equation*}
\dim(V) + \dim \bigl( G^1_k(\tilde C) \bigr)
\leq 2k-2
\end{equation*}
for general $C \in V$.
\end{thm}

\begin{thm}[{\cite[Theorem~3.1]{CK}}]
\label{t:CK} 
Let $(S,L)\in \mathcal K_p$ be such that $\Pic(S)= \Z[L]$,
and $C \in \gsev L g$; let $\delta=p-g$.
Let $r,d$ be nonnegative integers.
If there exists a $g^r_d$ on the normalization of $C$, then 
\[
\delta \geq 
\alpha \bigl( rg - (d-r)(\alpha r+1)
\bigr),
\quad \text{where} \quad
\alpha = 
\left\lfloor \frac {gr + (d-r)(r-1)} 
{2r(d-r)}
\right\rfloor.
\]
\end{thm}

\begin{thm}[{\cite{L,GL,BFT,Gomez}}]
\label{t:GL-BFT}
Let $(S,L)\in \mathcal K_p$ be such that $\Pic(S)= \Z[L]$,
and $C\in |L|$.
The Clifford index of $C$, computed with sections of rank one torsion
free sheaves on $C$ (see
\cite [p.~202]{DeSe} or \cite {BFT}), equals
$\lfloor \frac {p-1} 2 \rfloor$.
\end{thm}

\section{Standard components}
\label{S:standard}

\subsection{The nodal map}
Let $(S,L)\in \mathcal K_p$. For any positive integer $n$, we denote
by $S^{[n]}$ the Hilbert scheme of $0$-dimensional subschemes of $S$ of
length $n$. Recall that $S^{[n]}$ is smooth of dimension $2n$ (see \cite
{Fog}).

Consider the morphism
\[
\phi_{L,\delta}: \sev L \delta\longrightarrow S^{[\delta]}, 
\]
called the \emph{nodal map}, which maps a curve $C\in \sev L \delta$ to
the scheme $\Delta$ of its nodes, indeed $0$-dimensional of length
$\delta$. We
set $\Phi_{L,\delta}:={\rm Im}(\phi_{L,\delta})$. If $V$ is an
irreducible component of $\sev L \delta$, we set
\[
\phi_V:= 
\restr {\phi_{L,\delta}} V\,\,\, \text{and}\,\,\,\,
\Phi_V:= {\rm Im}(\phi_{V}).
\]
Let $\Delta$ be a general point in $\Phi_{V}$. Then $\phi_{V}^
{-1}(\Delta)$ is an open subset of the linear system
$|L(-2\Delta)|:=\P (\H^ 0(S, L \otimes \mathcal I^ 2_{\Delta,S}))$ of
curves in $|L|$ singular at $\Delta$. We set
\begin{equation*}\label{eq:dim1}
\dim (|L(-2\Delta)|)=p-3\delta+h_V,
\end{equation*}
which defines the non-negative integer $h_V$, called the \emph{excess}
of $V$. By Proposition \ref {prop:sevv}, one has
\begin{equation}\label{eq:dim2}
\dim(\Phi_{V})=2\delta-h_V.
\end{equation}
The following is immediate:
\begin{lem}
\label{lem:im}  
Let $(S,L)\in \mathcal K_p$, and let $V_1, V_2$ be two distinct
irreducible components of $\sev L \delta$. Then $\Phi_{V_1}$ and
$\Phi_{V_2}$ have distinct Zariski closures in $S^{[\delta]}$.
\end{lem}

\subsection{A useful lemma}

Let $C \in |L|$ be a reduced curve, and consider the conductor ideal
$A \subset \O_C$ of the normalization $\nu: \tilde C \to C$.
There exists a divisor $\tilde\Delta$ on $\tilde C$ such that 
$A=\nu_*\O_{\tilde C}(-\tilde\Delta)$, 
and one has
$\omega _{\tilde C} = \nu^* \omega_C \otimes \O_{\tilde
  C}(-\tilde\Delta)$.
It is a classical result that
$\nu^* |L\otimes A| = |\omega_{\tilde C}|$,
see \cite[Lemma~3.1]{DeSe}.
The same argument proves that
$\nu^* |L\otimes A^{\otimes 2}| = |\omega_{\tilde C}(-\tilde\Delta)|$.

Consider the particular case when $C$ has ordinary cusps
$p_1,\ldots,p_k$ and nodes $p_{k+1},\ldots,p_\delta$ as its only
singularities. Denote by $p_1,\ldots,p_k \in \tilde C$
the respective preimages of $p_1,\ldots,p_k \in C$ by the
normalisation $\nu$, abusing notations, 
and by $p_i'$ and $p_i''$ the two preimages
of $p_i$ for $i=k+1,\ldots,\delta$. 
Then $A$ is the product of the maximal ideals of 
$p_{1},\ldots,p_\delta$,
i.e., $A=\cI_{\Delta,S} \otimes \O_C$ with
$\Delta=\{p_1,\ldots,p_\delta\}$,
and
\[
\tilde\Delta = 2 \sum\nolimits _{i=1} ^k p_i 
+ \sum\nolimits _{i=k+1} ^\delta (p_{i}' + p_{i}'').
\]
The previous identity $\nu^* |L\otimes A^{\otimes 2}| =
|\omega_{\tilde C}(-\tilde\Delta)|$ readily implies the following.

\begin{lem}
\label{l:indep-adj^2}
Let $j$ be the closed immersion $C \hookrightarrow S$. 
One has
\[
(j\circ \nu)^* \bigl(\,
|L(-2\Delta)| \,\bigr)
= |\omega _{\tilde C} (-\tilde \Delta)|,
\]
and therefore
$
\dim \bigl(\, |L(-2\Delta)| \,\bigr)
= h^0 \bigl( \omega _{\tilde C} (-\tilde \Delta) \bigr)
$.
\end{lem}

\subsection{Standard components} 
Let $V$ be an irreducible component of $\sev L \delta$.  We call $V$
\emph{standard} if $h_V=0$. If $V$ is standard and $\Delta\in \Phi_V$
is general, then
\[
0\leqslant \dim(\phi_V^ {-1}(\Delta))=\dim
(|L(-2\Delta)|)=p-3\delta,
\] hence $p\geqslant 3\delta$. Moreover if
$V$ is standard, then $\dim(\Phi_V)=2\delta$, hence $\Phi_V$ is dense
in $S^{[\delta]}$. We will prove in Proposition \ref {prop:stand}
below that if $p\geqslant 3\delta$ and if ${\rm Pic}(S)=\Z[L]$, then
there is a unique standard component of $\sev L \delta$. To do this, we
need to recall some basic fact from \cite {CC}. 

Let $Y\subset \P^ N$ be an irreducible, $n$-dimensional,
non-degenerate, projective variety. Let $\mathcal H$ be the linear
system cut out on $Y$ by the hyperplanes of $\P^ N$, i.e.,
\[
\mathcal H=\P({\rm Im}(r)) 
\quad \text{where}\quad 
r: \H^ 0(\P^ N,\O_{\P^ N}(1))\to 
\H^ 0(Y,\O_Y\otimes \O_{\P^ N}(1))
\]
is the restriction map.
Let $k$ be a non-negative integer.  The variety $Y$ is said to be
$k$-\emph{weakly defective} if given $p_0,\ldots, p_k\in Y$ general
points, the general element of 
$\mathcal H(-2p_0-\ldots -2p_k)$
has a positive dimensional singular locus,
where $\mathcal H(-2p_0-\ldots -2p_k)$ denotes
the linear system  of divisors in $\mathcal H$ singular at
$p_0,\ldots, p_k$.

\begin{prop}[{\cite[Theorem~1.4] {CC}}]
\label{prop:wd}
Let $Y\subset\P^ N$ be an irreducible, $n$-dimensional,
non-degenerate, projective variety.  Let $k$ be a non-negative integer
such that $N\geq (n+1)(k+1)$.
If $Y$ is not $k$-weakly defective, then given $p_0,...,p_{k}$
general points on $Y$, one has:\\
\begin{inparaenum}
\item [(i)] $\dim({\mathcal H}(-2p_0-...-2p_{k}))=N-(n+1)(k+1)$;\\
\item [(ii)] the general divisor $H\in {\mathcal H}(-2p_0-...-2p_{k})$
has \emph{ordinary} double points at $p_0,...,p_{k}$, i.e., double
points with tangent cone of maximal rank $n$, and no other
singularity.
\end{inparaenum}
\end{prop}

In \cite [Theorem 1.3]{CC} one finds the classification of $k$-weakly
defective surfaces. After an inspection which we leave to the reader,
one sees that:

\begin{prop}
\label{prop:dk3} 
Let $(S,L)\in \mathcal K_p$ be such that $\Pic(S)=\Z[L]$, and assume
$p\geq 3$. Consider $S$ embedded in $\P^ p$ via the morphism
determined by $|L|$. Then $S$ is not $k$-weakly defective for any
non-negative integer $k$.
\end{prop}

We can therefore apply Proposition \ref {prop:wd} and conclude that:

\begin{prop}\label{cor:dk3}
Maintain the assumptions of Proposition~\ref{prop:dk3},
and let $\delta$ be a non-negative integer such that
$3 \delta \leq p$.
Then given $\Delta\in S^{[\delta]}$ general,
one has $\dim (|L(-2\Delta)|)=p-3\delta$ and the general curve in
$|L(-2\Delta)|$ has nodes at $\Delta$ and no other singularities.
\end{prop}

As a consequence we have:

\begin{prop}\label{prop:stand}
Under the assumptions of Proposition~\ref{cor:dk3},
there is a unique standard component
$\stsev L \delta$
of $\sev L \delta$, which is the unique irreducible component $V$ of
$\sev L \delta$ such that $\phi_V: V\to S^{[\delta]}$ is dominant.
\end{prop}
\begin{proof}  
Proposition \ref {cor:dk3} implies that there is a standard component $V$
of $\sev L \delta$ such that $\phi_V: V\to S^{[\delta]}$ is
dominant. By Lemma \ref {lem:im}, it is the unique standard
component. \end{proof}

\section{A lower bound on the excess} 
\label{S:estim-h}

\noindent
This section is entirely devoted to the proof of the following:

\begin{prop}
\label{lem:main1}
Let $p \geq 11$ and $\delta >1$, $(p,\delta) \neq (12,4)$,
be integers such that $3 \delta \leq p$. 
We consider $(S,L) \in \K_p$ such that $\Pic(S) = \Z[L]$.
For all non-standard component $V$ of
$\sev L \delta$, one has $h_V\geq 3$.
\end{prop}

Let $V$ be a non-standard component of $\sev L \delta$ as above. One
has $h_V >0$ by definition, and we shall proceed by contradiction to
show that $h_V$ may neither equal $1$ nor $2$. 

\subsection{Proof that $h_V \neq 1$}

In the setup of Proposition~\ref{lem:main1}, we assume by
contradiction that $h_V=1$.
Then the closure of $\Phi_V$ is an irreducible divisor in
$S^{[\delta]}$. Let $\Delta\in \Phi_{V}$ be a general point. 
It can be seen as the limit of a general 1-dimensional family
$\{\Delta_t\}_{t\in \D}$, where $\D$ is a complex disk, and $\Delta_t$ is
general in $S^{[\delta]}$ for $t\neq 0$. 
In particular, we may assume 
$\dim(\phi^{-1}_{L,\delta}(\Delta_t))=p-3\delta$ for $t\in \D-\{0\}$.
We define the limit $\mathcal L_\Delta$ of $\phi^
{-1}_{L,\delta}(\Delta_t)$ as $t\to 0$ as the fibre over $0 \in \D$
of the closure of 
$\bigcup _{t\neq 0} \bigl( \phi^ {-1}_{L,\delta}(\Delta_t) \bigr)$
inside $|L| \times \D$.
Then:\\
\begin{inparaenum}
\item $\mathcal L_\Delta$ is a
$(p-3\delta)$-dimensional sublinear system of $|L(-2\Delta)|$;\\
\item  $\mathcal L_\Delta$ is contained in 
$\overline V \cap \cstsev L \delta$;\\
\item  since $\sev L \delta$ is smooth, by (ii) the general
curve in $\mathcal L_\Delta$ does not belong to $\sev L \delta$, i.e.,
it has singularities worse than only nodes at the points of
$\Delta$;\\
\item  as $\Delta$ moves in a suitable dense open subset $U$ of
$\Phi_V$, the union $\bigcup _{\Delta\in U} \mathcal L_\Delta$ describes a
locally closed subset of dimension
\[
\dim(\Phi_V)+(p-3\delta)=(2\delta-1)+(p-3\delta)=g-1,
\]
which is dense in an irreducible component $W$ of 
$\overline V\cap \cstsev L \delta$, 
where $g=p-\delta$ as usual.
\end{inparaenum}

Let $C$ be the general curve in $W$, which belongs to $\mathcal
L_\Delta$ for some general $\Delta\in \Phi_V$. By (i) and (iii) above, $C$
is singular at $\Delta$ but it is not $\delta$-nodal. By Proposition
\ref {prop:dim} one has $p_g(C)\geqslant g-1$, hence $g-1\leqslant
p_g(C)\leqslant g$. We will show that each of these two possible
values leads to a contradiction, thus proving that $h_V\neq 1$.

\subsubsection{Case $p_g(C)=g-1$} 
Since $\dim(W)=g-1$, it follows from Proposition \ref {prop:sev} that
$W$ is dense in the closure of a component of $\sev L {\delta+1}$, i.e.,
$C$ is a $(\delta+1)$-nodal curve, with only one extra node
$p_{\delta+1}\not\in \Delta$. By Proposition \ref {prop:ext}, locally
around $C$ there is only one smooth branch $\mathcal V$ of $\csev L
\delta$ containing $W$ and such that when the general point of
$\tilde C$ of $\mathcal V$ specializes at $C$, then set of $\delta$
nodes of $\tilde C$ specializes at $\Delta$. This is a
contradiction, because both $\overline V$ and $\cstsev L \delta$ 
contain $W$. Therefore, it is impossible that
$p_g(C)=g-1$.

\subsubsection{Case $p_g(C)= g$}
Since $C$ is singular at $\Delta=p_1+\ldots+p_\delta$, it is singular
only there, and has only nodes and (simple) cusps (with local equation
$x^ 2=y^ 3$); it must have at least one cusp by (iii).

\begin{claim}\label{claim:1} $C$ has only one cusp. \end{claim}

\begin{proof} [Proof of the Claim] 
Suppose that $C$ has cusps at $p_1,\ldots, p_k$ and nodes at
$p_{k+1},\ldots,p_\delta$, with $k\geqslant 1$. The tangent space to
the equisingular deformations of $C$ in $S$ is $\H^ 0(C, L\otimes
\mathcal I\otimes \mathcal O_C)$, where 
$\mathcal I$ is the ideal sheaf associated
to the \emph{equisingular ideal} (see \cite [\S~3] {DeSe})
$I=\prod_{i=1}^ \delta I_{p_i}$, where:\\
\begin{inparaitem}
\item $I_{p_i}=(x, y^ 2)$, if the local equation of $C$ around $p_i$
is $x^ 2=y^ 3$, for $i=1,\ldots,k$;\\
\item  $I_{p_i}$ is the maximal ideal at $p_i$, for $i=k+1,\ldots,\delta$.
\end{inparaitem}

Let $\nu: \tilde C\to C$ be the normalization. 
We abuse notation and denote by
$p_1,\ldots, p_k$ their counterimages by $\nu$, whereas we denote by
$p_{i}'$ and $p_i''$ the two points of $\tilde C$ in the preimage of
$p_i$ by $\nu$, for $i=k+1,\ldots,\delta$.
By pulling back by $\nu$ the sections of 
$\H^ 0(C, L\otimes \cI\otimes \mathcal O_C)$
and dividing by sections vanishing at the fixed divisor 
$2\sum_{i=1}^ k p_i+\sum_{i=k+1}^ \delta(p_i'+p_i'')$ 
(see \cite [\S 3.3] {DeSe}), we find an isomorphism
\[
\nu^ *: \H^ 0(C, L\otimes \cI\otimes \mathcal O_C)
\cong 
\H^0(\tilde C,\omega_{\tilde C}(-p_1-\ldots- p_k)),
\]
hence
\begin{equation}\label{eq:in}
h^ 0(\tilde C,\omega_{\tilde C}(-p_1-\ldots- p_k))
= h^ 0(C, L\otimes \cI\otimes \mathcal O_C)
\geqslant \dim(W)=g-1.
\end{equation}
This implies that the points $p_1,\ldots,p_k$ are all identified by
the canonical map of $\tilde C$, 
which is possible only if either $k=1$, or $k=2$ and
$\dim(|p_1+p_2|)=1$.  We now prove that 
$\tilde C$ may not be hyperelliptic,
hence the latter case does not occur. 

By Theorem~\ref{t:KLM}, if $\tilde C$ is hyperelliptic then
$\dim(W) = g-1 \leq 2$. This contradicts our assumptions that $3\delta
\leq p$ and $p \geq 11$: indeed, as $g=p-\delta$ they imply that $g>3$.
Hence the only possibility left is that
$k=1$,
which proves the claim.
\end{proof}

Note moreover that since $k=1$, equality holds in \eqref{eq:in}.

Let $N_{C/S}\cong \restr L C$ be the normal bundle of $C$ in $S$. 
We have the exact sequence
\[
\textstyle
0\to N'_{C/S} \to N_{C/S} \to T^ 1_C\cong 
\mathcal O_{p_1}^2 \oplus \bigoplus_{i=2}^ \delta \mathcal O_{p_i}  \to 0
\]
where $N'_{C/S}$ is the \emph{equisingular normal sheaf} of $C$ in
$S$, and one has
$
N'_{C/S}\cong N_{C/S}\otimes \mathcal I 
$.
So $\H^ 0(C,N'_{C/S})=\H^ 0(C, L\otimes \mathcal I\otimes \mathcal O_C)$
is the tangent space to the equisingular deformations of $C$ in $S$.

We have $h^ 0(C,N_{C/S})=p$ and, as we saw, $h^
0(C,N'_{C/S})=g-1=p-\delta-1$. Thus the map
\begin{equation}\label{eq:T}
\H^ 0(C,N_{C/S})\to T^ 1_C
\end{equation}
is surjective, and
$\H^1(C,N'_{C/S})\cong \H^ 1(C,N_{C/S})\cong \C$.
Moreover the obstruction space to deformations of $C$ in $S$,
contained in $\H^ 1(C,N_{C/S})$, is zero as is well-known (see, e.g.,
\cite[\S~4.2]{DeSe}).
This implies that, locally around $C$, $\csev L \delta$ is the
product of the equigeneric deformation spaces inside the versal
deformation spaces of the singularities of $C$.  By looking at the
versal deformation space of a cusp (see, e.g., \cite [p. 98] {HM}), we
deduce that $\csev L \delta$ has a double point at $C$ with a
single cuspidal sheet.  This is a contradiction, because we assumed
that both $\overline V$ and $\cstsev L \delta$ contain
$C$. This contradiction proves that $p_g(C)=g$ cannot occur.\medskip

In conclusion we have proved that if $h_V= 1$ then
$p_g(C)$ equals either $g-1$ or $g$, but both these possibilities lead
to contradictions, hence $h_V \neq 1$.

\subsection{Proof that $h_V \neq 2$}
Still in the setup of Proposition~\ref{lem:main1}, we now assume by
contradiction that $h_V=2$. 
Then $\dim(\Phi_V)=2\delta-2$. Let $\Delta\in \Phi_{V}$ be a general
point. Again $\Delta$ can be seen as the limit of general
1-dimensional families $\{\Delta_t\}_{t\in \D}$, where $\D$ is a disk,
and $\Delta_t$ is general in $S^{[\delta]}$ for $t\neq 0$. We consider
the closure $\mathcal L_\Delta$ of the union of all
$(p-3\delta)$-dimensional sublinear systems 
$\lim _{t \to 0} \bigl( \phi^ {-1}_{L,\delta}(\Delta_t) \bigr) \subset
|L(-2\Delta)|$ 
as $\{\Delta_t\}_{t\in \D}$ varies among all families as above.
Similarly to the case $h_V=1$, we have:\\
\begin{inparaenum}
\item [(i)] $\mathcal L_\Delta$ is contained in $\overline V\cap
\cstsev L \delta$ and $\dim(\mathcal L_\Delta)=
p-3\delta+\epsilon$, with $0\leqslant \epsilon\leqslant 1$ ;\\
\item [(ii)] the general curve in $\mathcal L_\Delta$ is singular at
$\Delta$ but has singularities worse than only nodes at the points of
$\Delta$;\\
\item [(iii)] as $\Delta$ moves in a suitable dense open subset $U$ of
$\Phi_V$, the union $\bigcup _{\Delta\in U} \mathcal L_\Delta$ describes a
locally closed subset of dimension
\[
\dim(\Phi_V)+\dim(\mathcal L_\Delta)=g-2+\epsilon,
\]
which is dense in an irreducible component $W$ of 
$\overline V\cap \cstsev L \delta$.
\end{inparaenum}

If $\epsilon=1$, then $\dim(W)=g-1$ and the discussion goes as in the
case $h_V=1$. So we assume $\epsilon=0$, hence $\dim(W)=g-2$. Let $C$
be the general curve in $W$. By Proposition \ref {prop:dim}, we have
$g-2\leqslant p_g(C)\leqslant g$.
We will prove that this cannot happen, thus proving that
$h_V\neq 2$. The proof parallels the one for $h_V\neq 1$.

\subsubsection{Case $p_g(C)=g-2$} 
By Proposition~\ref {prop:sev},
$C$ is a $(\delta+2)$-nodal curve, with two extra nodes $p_{\delta+1},
p_{\delta+2}\not\in \Delta$ and  $W$ is dense in the closure of a
component of $\sev L {\delta+2}$. By Proposition~\ref {prop:ext}, locally
around $C$ there is only one smooth branch $\mathcal V$ of $\csev L
\delta$ containing $W$ and such that when the general point of
$C'$ of  $\mathcal V$ specializes at $C$, then set of $\delta$
nodes of $C'$ specializes at $\Delta$. This is a contradiction,
because both $\overline V$ and $\cstsev L \delta$
contain $W$. Hence $p_g(C)=g-2$ cannot happen.
\endproof 

\subsubsection{Case $p_g(C)=g-1$}
\label{g-1}
In this case we have the two following disjoint possibilities for $C$:\\
\begin{inparaenum}[(a)]
\item $C$ has precisely one more singularity $p_0$ besides the ones in
  $\Delta$;\\ 
\item $C$ has no singularities besides the ones in $\Delta$, 
either an ordinary tacnode or a ramphoid cusp 
(with local equation $x^2=y^{4+\epsilon}$,
$\epsilon =0$ or $1$ respectively) at one of the points
of $\Delta$, 
and nodes or ordinary cusps at the other points of $\Delta$.
\end{inparaenum}

\medskip
\par\noindent {\itshape Subcase (a)}. 
The points $p_0,\ldots, p_\delta$ are either nodes or cusps. 
Arguing as for Claim~\ref{claim:1},
we see that at most one of these points can be a cusp. 

If $C$ is $(\delta+1)$-nodal, then $W$ sits in an irreducible
component of $\csev L {\delta+1}$, and we get a contradiction as
in the proof of case $p_g(C)=g-1$ for $h_V=1$.

If $C$ is $\delta$-nodal and 1-cuspidal, then again the map \eqref
{eq:T} is surjective and the deformation space of $C$ is locally the
product of the versal deformation spaces at $p_0,\ldots, p_\delta$.
We then have the two following possibilities.

If $p_0$ is a node, then $W$ sits in a $(g-1)$-dimensional irreducible
variety $W'$ parametrizing curves which are $(\delta-1)$-nodal and
1-cuspidal, such that when the general member of $W'$ tends to $C$, its
singularities tend to $\Delta$. Moreover the map \eqref {eq:T}
is surjective for the general member of $W'$.
Then $W'$ should be contained in both $\overline V$ and 
$\cstsev L \delta$. On the other hand, as usual by now, $\csev L
\delta$ should be unibranched along $W'$, a contradiction.

If $p_0$ is the cusp, then $W$ sits in a $(g-1)$-dimensional
irreducible component $W'$ of $\csev L {\delta+1}$, such that
when the general member of $W'$ tends to $C$, its singularities
tend to $p_0,\ldots, p_\delta$. By Corollary \ref {cor:incl}, $W'$
should be contained in both $\overline V$ and $\cstsev L \delta$,
leading again to a contradiction. \medskip

\medskip
\par\noindent {\itshape Subcase (b)}. 
Suppose the tacnode 
or ramphoid cusp
is located at $p_1$, that $p_2,\ldots, p_k$ are
cusps, and $p_{k+1},\ldots,p_\delta$
are nodes:
one has $1\leqslant k\leqslant \delta$, and $k=1$ (resp.\ $\delta$)
means that there is no cusp (resp.\ no node).
If $C$ has local equation 
$x^ 2=y^ {4+\epsilon}$
around
$p_1$, then the equisingular ideal $I_{p_1}$ at $p_1$ is 
$(x,y^ {3+\epsilon})$
(see \cite [\S 3] {DeSe}).  As usual set $I=\prod_{i=1}^ \delta
I_{p_i}$ and let $\mathcal I$ be the corresponding ideal sheaf.

We have
\begin{equation}\label{eq:gen}
h^ 0(C,N'_{C/S})=h^ 0(C, N_{C/S}\otimes \mathcal I)\geqslant \dim(W)=g-2.
\end{equation}
Now we can look at $\H^ 0(C,N'_{C/S})$ as defining a linear series of
\emph{generalized divisors} on the singular curve $C$ (see \cite
{Hart} and \cite [\S 3.4] {DeSe}). Then $N'_{C/S}= N_{C/S}\otimes
\mathcal I\cong \omega_C(-E)$ where $E$ is the effective generalized
divisor on $C$ defined by the ideal sheaf $\mathcal I$ and \eqref
{eq:gen} reads
\begin{equation}\label{eq:cc}
h^ 0(C, \omega_C(-E))\geqslant g-2.
\end{equation}
The subscheme of $C$ defined by $\mathcal I$ has length 
$3+\epsilon$
at the tacnode, length 2 at each cusp and length 1 at the nodes, so that
\[
\deg(E)=3 +\epsilon +2(k-1)+\delta-k=\delta+k+1
 +\epsilon.
\]
By Riemann--Roch and Serre duality 
\cite[Theorems~1.3 and 1.4]{Hart}, one has
\begin{equation}\label {eq:bb}
h^ 0(C, \omega_C(-E))=h^ 1(C, \mathcal O_C(E))=h^ 0(C,  \mathcal
O_C(E))-\deg(E)+p-1=h^ 0(C,  \mathcal O_C(E))+g-k-2
 -\epsilon.
\end{equation}
Next we argue as in the proof of \cite [Prop.~4.8] {DeSe}. 
If $h^1(C, \mathcal O_C(E))<2$, then by \eqref {eq:cc} we have 
$g\leq 3$, 
which contradicts our assumptions that $3\delta \leq p$ and $\delta
>1$.
If on the other hand $h^ 0(C, \mathcal O_C(E))<2$, then by \eqref
{eq:cc} and \eqref {eq:bb} we have
\[
g-2\leqslant h^ 1(C, \mathcal O_C(E))\leqslant g-k-1
-\epsilon,
\]
hence $\epsilon=0$ and
$k=1$, i.e., the singularities of $C$ are precisely one ordinary
tacnode and $\delta-1$ nodes.
There is then equality in both
\eqref {eq:gen} and \eqref {eq:cc}, hence once more \eqref {eq:T} is
surjective and the deformation space of $C$ is locally the product of
the versal deformation spaces at $p_1,\ldots, p_\delta$. By looking at
the versal deformation space of a tacnode
(see \cite [p.~181]{CH}) we
see that $W$ is contained in $\csev L \delta$ which should be
unibranched along $W$, a contradiction.

So one has necessarily that $h^ i(C, \mathcal O_C(E))\geqslant 2$, for
$i=1,2$. 
Then, since ${\rm Cliff}(C)=\lfloor\frac {p-1} 2\rfloor$ by
Theorem~\ref{t:GL-BFT}, one has
\[
\textstyle
p+1-h^ 0(C, \mathcal O_C(E))-h^ 1(C, \mathcal O_C(E))
=\deg(E)-2h^ 0(C, \mathcal O_C(E))+2
\geq \lfloor\frac {p-1} 2\rfloor
\]
hence
\[
\textstyle
g-2\leqslant h^ 1(C, \mathcal O_C(E))
\leqslant p+1 -\lfloor\frac {p-1} 2\rfloor
- h^ 0(C, \mathcal O_C(E))
\leqslant p-1- \lfloor\frac {p-1} 2\rfloor
=\lceil \frac {p-1} 2\rceil.
\]
Plugging in the inequality $3\delta \leq p$, one finds
\begin{equation}
\label{ineq:cliff}
\frac 2 3 p -2
\leq p-\delta-2
=g-2
\leq \lceil \frac {p-1} 2\rceil
\leq \frac p 2
\end{equation}
which implies $p \leq 12$, hence $p=11$ or $12$.
Case $p=11$ is impossible by \eqref{ineq:cliff}, since there is no
integer between the two extremes in \eqref{ineq:cliff}. If $p=12$,
then \eqref{ineq:cliff} implies $g=8$, hence $\delta=4$, which is
excluded by assumption. Hence subcase~(b) is impossible. This
concludes the proof that $p_g(C) \neq g-1$.

\subsubsection{Case $p_g(C)=g$} 
As in the case $h_V=1$, $C$ is
singular only at $\Delta=p_1+\ldots+p_\delta$, having only nodes and
simple cusps, and it must have at least one cusp.

\begin{claim}\label{claim:2} $C$ has at most two cusps. \end{claim}

\begin{proof} [Proof of the Claim] The proof goes as the one of
Claim~\ref {claim:1}, from which we keep the notation. If $C$ has
cusps at $p_1,\ldots, p_k$, we have
\begin{equation}\label{eq:cusp}
h^ 0(\tilde{C},\omega_{\tilde{C}}(-p_1-\ldots- p_k))\geqslant \dim(W)=g-2.
\end{equation}

We argue by contradiction and assume $k\geqslant 3$. As in the proof
of Claim~\ref {claim:1}, we see that $\tilde{C}$ is not
hyperelliptic:
this would imply by Theorem~\ref{t:KLM} that 
$g-2 =\dim(W) \leq 2$, hence $p=6$ and $g=4$; but in this case
$\delta=2$ and since $k \leq \delta$ we are out of the range $k
\geqslant 3$.

The only other possibility is that $\tilde{C}$ is
trigonal, $k=3$, and $\dim (|p_1+p_2+p_3|)=1$. In this case, one would
have $g-2 =\dim(W) \leq 4$ by Theorem~\ref{t:KLM}, which 
together with the inequality $p \geq 3\delta$ implies that $p \leq 9$:
This is in contradiction with our assumptions. It is thus impossible
that $k \geq 3$, and the Claim is proved.
\end{proof}

By Claim~\ref {claim:2}, we have only the following two mutually
disjoint possibilities:\\
\begin{inparaenum}[(a)]
\item $C$ has precisely one cusp at $p_1$, and 
$h^0(\tilde{C},\omega_{\tilde{C}}(-p_1))=g-1> g-2 = \dim(W)$;\\
\item  $C$ has precisely two cusps at $p_1$ and $p_2$, and 
$h^ 0(\tilde{C},\omega_{\tilde{C}}(-p_1- p_2))=g-2= \dim(W)$.
\end{inparaenum}

\medskip
\par\noindent {\itshape Subcase (a)}. 
We have $h^ 0(C,N'_{C/S})=h^
0(\tilde{C},\omega_{\tilde{C}}(-p_1))=g-1$, hence the map \eqref
{eq:T} is surjective.
This implies as in the case $h_V=1$ and $p_g=g$
that $W$ is contained in a
subvariety $W'$ of dimension $g-1$ contained in 
$\csev L \delta$, whose general point corresponds to a curve which has
$\delta-1$ nodes and one cusp, and, as in the proof of case $h_V=1$,
$\csev L \delta$ is unibranched locally at any point of $W'$
corresponding to such a curve for which the map \eqref {eq:T} is
surjective. This contradicts the fact that
$W$ is an irreducible component of 
$\overline V\cap \cstsev L \delta$.

\medskip
\par\noindent {\itshape Subcase (b)}. 
In this case $W$ is dense in the equisingular deformation locus of $C$
and again the map \eqref {eq:T} is surjective.  This again implies that
$\csev L \delta$ is unibranched locally around $C$, which
leads to a contradiction.

\par\medskip
This concludes the proof that $h_V \neq 2$, hence also the proof of
Proposition~\ref{lem:main1}.

\section{Proof of Irreducibility if $p>4\delta -4$}
\label{S:conclusion}

In this section we conclude the proof of Theorem~\ref{t:main}.
So let $(S,L)$ be a primitively polarized $K3$ surface of genus $p
\geq 11$ such that $\Pic(S)=\Z [L]$, and $\delta$ be a non-negative
integer such that $4\delta-3 \leq p$.

These assumptions imply that $p \geq 3\delta$, so that the notion of
standard 
component makes sense, and the Severi variety $\sev L \delta$ has a
unique standard component by Proposition~\ref{prop:stand}. 
We assume by contradiction that $\sev L \delta$ is not irreducible:
this means that there exists a non-standard
component $V$ of the Severi variety $\sev L \delta$, and we shall see
this contradicts the inequality $p>4\delta -4$. 

Let $h=h_V$. If $\delta \leq 1$, then Theorem~\ref{t:main} is
trivial; else we're in the range of application of
Proposition~\ref{lem:main1}
(note that the case $(p,\delta)=(12,4)$ is excluded by the hypothesis
$p\geq 4\delta-3$), 
hence $h \geq 3$.

Consider a general member $C \in V$, and let 
$\Delta = \{p_1, \ldots, p_\delta \}$
be the set of its nodes. Let
$\nu: \tilde C \to C$ be the normalization map, and for all 
$i=1,\ldots,\delta$, $p_i'$ and $p_i''$ the two antecedents of $p_i$ by
$\nu$. We consider the divisor
$\tilde \Delta = \sum _{i=1} ^\delta
(p_i'+p_i'')$
on $\tilde C$.

\begin{lem}
\label{l:dim-deltatilde}
The complete linear series $|\tilde \Delta|$ is a 
$g^h _{2\delta}$.
\end{lem}

\begin{proof}
One has $h^1(\tilde \Delta)=p-3\delta+h$ by Lemma~\ref{l:indep-adj^2},
and then the result follows from the Riemann--Roch formula.
\end{proof}

\begin{proof}[Conclusion of the proof of Theorem~\ref{t:main}]
We maintain the above setup.
We first apply Theorem~\ref{t:CK}: 
Let $g=p-\delta$ denote the geometric genus of $C$, and set
\begin{equation*}
\alpha =
\left\lfloor
\frac
{gh + (2\delta-h)(h-1)}
{2h(2\delta-h)}
\right\rfloor
=
\left\lfloor
\frac g {2(2\delta-h)}
+ \frac {h-1} {2h}
\right\rfloor ;
\end{equation*}
the existence of a $g^h_{2\delta}$ on $\tilde C$ implies the 
inequality
\begin{align}
\label{eq:CK}
\alpha hg +\alpha h(\alpha h +1) 
& \leq
\delta (2\alpha^2 h + 2\alpha +1).
\end{align}

Let us also apply
Theorem~\ref{t:KLM}:
The existence of a
$g^h_{2\delta}$ on $\tilde C$ induces the existence of a family
of dimension $2(h-1)$ of
$g^1_{2\delta}$'s on $\tilde C$, parametrizing the lines in the
$g^h_{2\delta}$, so it holds that
\begin{equation*}
\dim (V) + \dim \bigl( G^1_{2\delta}(\tilde C) \bigr) \geq g + 2(h-1),
\end{equation*}
which implies by Theorem~\ref{t:KLM} that
\begin{align}
g &\leq 2(2\delta-h).
\label{ineq:KLM}
\end{align}

Inequality~\eqref{ineq:KLM} implies that 
\begin{equation*}
\alpha
= \left\lfloor
\frac g {2(2\delta-h)}
+ \frac {h-1} {2h}
\right\rfloor
\leq \left\lfloor
1+\frac 1 2
\right\rfloor 
= 1.
\end{equation*}
Let us now show by contradiction that $\alpha=1$. If $\alpha  \leq 0$,
then 
\begin{equation*}
\frac
{gh + (2\delta-h)(h-1)}
{2h(2\delta-h)}
<1
\iff 
g <
(2\delta-h)(1+\frac 1 h) 
\iff
p <
\delta (3+ \frac 2 h)-h-1;
\end{equation*}
plugging in the inequality $h \geq 3$, we get that $\alpha\leq 0$ implies 
$p <
\frac {11} 3 \delta -4$,
in contradiction with our assumption that $p>4\delta-4$. Hence
$\alpha=1$. 

Therefore, \eqref{eq:CK} gives the inequalities
\begin{equation*}
hg + h( h +1) 
\leq
\delta (2 h + 3)
\iff
p \leq
\delta (3 + \frac 3 h)
-h-1.
\end{equation*}
Taking into account the fact that $h \geq 3$,
this implies that $p \leq 4\delta-4$.
In conclusion, the existence of a non-standard component of
$\sev L \delta$ is in contradiction with the inequality
$p > 4\delta-4$.
\end{proof}

\providecommand{\bysame}{\leavevmode\hboxto3em{\hrulefill}\thinspace}

\renewcommand{\thefootnote}{}
\makeatletter
\footnotetext{\@acks}
\makeatother

\end{document}